%% file: main.tex
\documentclass[10pt,twoside,reqno]{amsart}
\usepackage[all]{xy}
        \usepackage {amssymb,latexsym,amsthm,amsmath,mathtools,dsfont,mathrsfs}
        \usepackage{enumitem,color}

\usepackage[
  top=1in,
  bottom=1in,
  left=1.2in,
  right=1.2in
]{geometry}
\usepackage{makecell}
\usepackage{mathtools}

\usepackage{longtable}
\usepackage{framed}

\usepackage[backend=bibtex,backref=true,maxbibnames=999,doi=false,isbn=false,url=false
]{biblatex} 
\usepackage{float}
\usepackage{hyperref}
\long\def\symbolfootnote[#1]#2{\begingroup%
\def\thefootnote{\fnsymbol{footnote}}\footnote[#1]{#2}\endgroup}

\makeatletter
\def\imod#1{\allowbreak\mkern10mu({\operator@font mod}\,\,#1)}
\makeatother
\makeatletter
\renewcommand*\env@matrix[1][*\c@MaxMatrixCols c]{%
  \hskip -\arraycolsep
  \let\@ifnextchar\new@ifnextchar
  \array{#1}}
\makeatother
\usepackage[capitalize,nameinlink]{cleveref} 
\usepackage{comment} 
\usepackage{xcolor} 
\hypersetup{
    colorlinks,
    linkcolor={red!80!black},
    citecolor={green!80!black},
    urlcolor={blue!80!black}
}

\newtheorem{theorem}{Theorem}[section]
\newtheorem{lemma}[theorem]{Lemma}
\newtheorem{corollary}[theorem]{Corollary}

\newtheorem*{theorem*}{Theorem}
\theoremstyle{definition}
\newtheorem{definition}[theorem]{Definition}
\newtheorem{remark}[theorem]{Remark}
\newtheorem{example}[theorem]{Example}
\newtheorem{problem}[theorem]{Problem}

\numberwithin{equation}{section}
\newcommand{\ignore}[1]{}

\newcommand{\mynote}[1]{}
\title[Hayashi Property for conjugation quandles]{Hayashi Property for conjugation quandles over finite Coxeter groups}
\author[Dilpreet Kaur]{Dilpreet Kaur}
\email{dilpreetkaur@iitj.ac.in}
\address{Indian Institute of Technology Jodhpur
N.H. 62, Nagaur Road, Karwar Jodhpur 342030
Rajasthan}
\author[Uday Bhaskar Sharma]{Uday Bhaskar Sharma}
\email{udaybhaskar.sharma@ddn.upes.ac.in}
\address{School of Advanced Engineering, UPES, Dehradun -248007, India}
\author[Pushpendra Singh]{Pushpendra Singh}
\email{pushpendra@iisermohali.ac.in}
\address{Indian Institute of Science Education and Research Mohali, Sector 81, Mohali 140306, India}
\thanks{}
\date{\today}
\begin{document}
\setcounter{section}{0}
\begin{abstract}
In this article, we show that for all finite irreducible Coxeter groups $G$ with a conjugation closed subset $C$, the conjugation quandle $\rm{Conj}(G,C)$ satisfies the Hayashi property.
\end{abstract}
\maketitle
\input{prelim}

\input{Sn}

\input{Exceptional}
\printbibliography
\end{document}

%% file: prelim.tex
\section{Introduction}
Quandles are algebraic structures introduced in \cite{Joy82a, Mat82} to construct classifying invariant of knots and links. Quandles have applications in various parts of algebra, such as set-theoretical solutions of the Yang-Baxter equations \cite{drinfeld2006some}, Hopf algebras \cite{andruskiewitsch2003racks}, construction of powerful invariants using quandle cohomology \cite{carter2003quandle}, etc. A quandle can be understood in a purely combinatorial way as an algebraic structure that is idempotent, and right translation maps are automorphisms. The applications of quandles motivate the development of a rich algebraic theory of quandles, such as the classification of finite simple quandles \cite{Joy82b}, the enumeration of quandles of up to order 14\cite{vojtvechovsky2019enumeration}, commutator theory of quandles\cite{bonatto2021commutator}, the enumeration of connected quandles of order up to 47\cite{hulpke2016connected}, etc. 

For a finite quandle, the set of cycle structures of all right translation maps is called a \textit{profile}. The profile of finite quandles has been studied extensively in the literature. We refer to \cite{ Lages_Lopes_2019,lages2024conjecture} for related work. We call a permutation $\sigma \in S_n$ a regular permutation if its disjoint cycle decomposition contains a cycle with length equal to the order $o(\sigma)$. Hayashi in \cite{Hay13} conjectured that all right translation maps of a connected quandle are regular permutations. The conjecture is proven for connected quandles with profile length at most five in \cite{lages2024conjecture}. Furthermore, the authors in \cite{Spa25} reduce the conjecture to the class of faithful connected quandles. The property of this conjecture is generalised for all finite quandles in \cite{Filip} by the name of \textit{Hayashi property} and studied for the conjugation quandles ${\rm Conj}(G,C)$ for a conjugacy class $C$ of group $G$.

In this article, we study the profile and, therefore, the Hayashi property of conjugation quandles on finite irreducible Coxeter groups. We prove that ${\rm Conj}(G,C)$ satisfies the Hayashi property for all finite irreducible Coxeter groups $G$ and conjugacy closed subset $C \subseteq G$. The article is organized as follows. In section \ref{section.prelim}, we introduce quandles and provide preliminary theory. In section \ref{section.B_n}, we prove our main result for finite Coxeter groups of type $A_n, I_2(n), B_n, C_n$ and $D_n$. In section \ref{section.exceptional}, we prove our main result for the finite Coxeter groups of type $E_6, E_7, E_8, F_4, H_3$ and $H_4$.

\section{Preliminaries}\label{section.prelim}

\begin{definition}
A quandle is a set $X$ equipped with a binary operation $\triangleright
$ satisfying the following axioms:
\begin{itemize}
    \item[\rm{(i)}.] $x\triangleright x=x$ for all $x\in X$
    \item[\rm{(ii)}.] The mapping $R_x:X\to X, y\mapsto y\triangleright x$ is an automorphism of $(X,\triangleright)$ for all $y \in X$.
\end{itemize}
\end{definition}
Equivalently, we can say that a quandle is an idempotent, right self-distributive, right quasigroup. If we exclude the axiom \rm{(i)}, then $(X,\triangleright)$ is known as \textit{rack}. A quandle is called a \textit{kei} if $R_x$ has order $2$ for all $x \in X$.

In the following, we give some examples of quandles.
\begin{itemize}
\item 
Affine quandle: Let $A$ be an abelian group and $f\in {\rm Aut}(A)$. The affine quandle ${\rm Aff}(A,f)$ is the set $A$ with binary operation $x\triangleright y=f(x)+(1-f)(y)$ for all $x,y\in A$.
\item Projection quandle: A set $X$ with binary operation $x\triangleright y=x$ is known as the projection quandle.
\item 
Conjugation quandle: Let $G$ be a group and $C$ be the conjugation closed subset of $G$. Then the conjugation quandle ${\rm Conj}(G,C)$ is the set $C$ with binary operation $x\triangleright y=yxy^{-1}$. If $C=G$, then we use notation ${\rm Conj}(G)$.
\item 
Coset quandle: Let $G$ be a group, $f\in {\rm Aut}(G)$, $H$ be a subgroup of $C_{G}(f)=\{g\in G: f(g)=g\}$. Then the coset quandle $Q(G,H,f)$ is the set of left cosets $G/H$ with binary operation $xH\triangleright yH=yf(y^{-1}x)H$.
\end{itemize}
In the above examples, we note that quandles are closely associated with groups. Joyce in \cite{Joy82a} studied various ways to represent a quandle using groups. In fact, he proved that every quandle is isomorphic to the disjoint union of coset quandles \cite[Theorem 7.2]{Joy82a}. 

For a quandle $(X,\triangleright)$, we can construct another quandle $(X,\triangleright^{-1})$ with $x\triangleright^{-1}y=R_y^{-1}(x)$. This is known as the \textit{dual} of $X$. For example, the \textit{dual} of ${\rm Conj}(G,C)$ is $(C,\triangleright^{-1})$ where $x\triangleright^{-1}y=y^{-1}xy$. 

We use ${\rm Sym}(X)$ to denote the symmetric group on the set $X$. Given a quandle $(X,\triangleright)$, we can construct a subgroup ${\rm RMlt}(X)$ of ${\rm Sym}(X)$ generated by the set of right translations. 
$${\rm RMlt}(X)=\langle R_x:x\in X \rangle \leq {\rm Sym}(X)$$
The group ${\rm RMlt}(X)$ acts on the quandle $(X,\triangleright)$ in a natural way. If this action is transitive, then the quandle is called a connected quandle. Furthermore, if the quandle homomorphism $X \to {\rm Conj}(G), x \mapsto R_x$ is injective, then $(X,\triangleright)$ is called a faithful quandle.

\begin{definition}
Let $(X,\triangleright)$ be a quandle. Let $\lambda_1^{\alpha_1}\lambda_2^{\alpha_2}\hdots\lambda_t^{\alpha_t}$ be the cycle structure of $R_x \in {\rm RMlt}(X)$ with $\lambda_i<\lambda_j$ for $i<j$. Let $L_x$ denote the set $\{\lambda_1,\lambda_1,\hdots,\lambda_1,\lambda_2,\lambda_2,\hdots,\lambda_2,\hdots,\lambda_t,\lambda_t\hdots,\lambda_t\}$. Then the set $\{L_x: x\in X\}$ is called the profile of a quandle.
\end{definition}

\begin{example}
Let $(X,\triangleright)$ be the quandle with the following table.
$$
\begin{array}{c|c c c c c} 
\triangleright \ &  1 & 2 & 3 & 4 & 5\\ \hline
\ 1 &  1 & 1 & 1 & 1 & 1 \\
\ 2 &  3 & 2 & 2 & 3 & 3 \\
\ 3 &  2 & 3 & 3 & 2 & 2 \\
\ 4 &  5 & 4 & 4 & 4 & 4 \\
\ 5 &  4 & 5 & 5 & 5 & 5 \\

\end{array}$$
Then ${\rm Prof}(X)=\{ \{1,2,2\},\{1,1,1,1,1\},\{1,1,1,1,1\},\{1,1,1,2\}, \{1,1,1,2\} \}$.
\end{example}

We note that for a quandle $(X,\triangleright)$, we have $R_yR_{x}R_{y}^{-1}=R_{x\triangleright y}$ for all $x,y\in X$. Hence, if $(X,\triangleright)$ is a connected quandle, then all right translation maps have the same cycle structure as elements of ${\rm Sym}(X)$. Therefore, for a connected quandle, the profile is defined to be the set $L_x$ for each $x\in X$. 

Let ${\rm Conj}(G,C)$ be the conjugation quandle where $C$ is the conjugacy class of $G$. Then for $x,y\in C$ there exist $g\in G$ such that $y=gxg^{-1}$ and so $R_y=R_{gxg^{-1}}=\phi_g|_CR_x\phi_{g^{-1}}|_C$, where $\phi_g:G\to G,h\mapsto ghg^{-1}$ is the inner automorphism of $G$. This shows that all right translations have the same cycle structure. Therefore, we define the profile as the set $L_x$ for any $x\in C$.

\begin{definition} Let $S_n$ denote the symmetric group of degree $n.$
Let $\lambda_1^{\alpha_1}\lambda_2^{\alpha_2}\hdots\lambda_t^{\alpha_t}$ be the cycle structure of $ g \in S_n$ with $\lambda_i<\lambda_j$ for $i<j$. We say $g$ is a regular permutation if $\lambda_i$ divides $\lambda_t$ for all $1\leq i\leq t$.
\end{definition}

\begin{definition}
Let $(X,\triangleright)$ be a finite quandle. Then it is said to satisfy the Hayashi property if all right translation maps of $X$ are regular permutations. Equivalently, we say that for every $L_x=\{1,l_1,l_2,\hdots,l_{t_x}\}$  with $1\leq l_1\leq l_2\leq \hdots \leq l_{t_x}$, $l_{t_x}$ is a multiple of $l_i$ for every $1\leq i\leq t_x$.
\end{definition}

The above definition is introduced in \cite{Filip}. It generalizes the Hayashi conjecture \cite[Conjecture 1.1]{Hay13} which states that for a connected quandle with profile $\{1,l_1,l_2,\hdots,l_t\}$ with $1\leq l_1\leq l_2\leq \hdots \leq l_t$, $l_i$ divides $l_t$ for every $1\leq i\leq t$.

\begin{lemma}\cite[Corollary 4.2]{Filip} \label{filip}
Let $G$ be a finite group, and let $C$ be a conjugacy class of $G$. Then ${\rm Conj}(G,C)$ satisfies the Hayashi property if there exist $x,y\in C$ such that $\langle x \rangle \cap C_G(y) \subseteq Z(\langle C \rangle)$.
\end{lemma}

We note that the above lemma can be generalized for the union of conjugacy classes as well.

\begin{lemma} \label{ue}
Let $G$ be a finite group,  and $C=\bigcup\limits_{i}C_i$ be the union of distinct conjugacy classes of $G$. Then ${\rm Conj}(G,C)$ satisfies the Hayashi property if for all $x\in C$ there exists $y\in C$ such that $\langle x \rangle \cap C_G(y) \subseteq Z(\langle C \rangle)$.
\end{lemma}
\begin{proof}
The proof is similar to that of Lemma \ref{filip}.
\end{proof}

We note that if ${\rm Conj}(G,C)$ and ${\rm Conj}(H,C')$ satisfy the Hayashi property for all conjugation closed subsets, then it is not necessary that ${\rm Conj}(G\times H,D)$ also satisfies the Hayashi property for all conjugation closed subsets of $G\times H$. As we have the following example.
\begin{example}\label{example}
We can take $G=H=S_3$, $C=\{(1),(1,2),(2,3),(1,3)\}$, $C'=\{(1,2),(2,3),(1,3),(1,2,3)\allowbreak,(1,3,2) \}$ and $D=\{((1),(1,2)),((1),(2,3)),\allowbreak((1),(1,3)),((1,2),(1,2,3))\allowbreak,((1,2),(1,3,2)),((2,3),(1,2,3))\allowbreak,((2,3),(1,3,2)),((1,3),(1,2,3)),((1,3),(1,3,2)) \}$. Then $Z(\langle D \rangle)=\{(1)\}$. For $x=((1,2),(1,2,3))$, there is no $y\in D$ such that $\langle x \rangle \cap C_{G\times H}(y)=\{(1)\}$.
\end{example}
However, we have the following result.

\begin{lemma}
If ${\rm Conj}(G,C)$ and ${\rm Conj(H,C')}$ satisfy the Hayashi property for conjugation closed subsets then ${\rm Conj}(G\times H,C\times C')$ satisfies the Hayashi property.
\end{lemma}

\begin{proof}
Let $C=\bigcup\limits_{i} C_i$ be the union of conjugacy classes of $G$. By assumption, for all $x_i \in C_i$, we have that $R_{x_i}$ is regular permutation in ${\rm Sym}(C)$. Similarly, for $C'=\bigcup\limits_j C_j' \subseteq H$, $R_{y_j}$ is regular for all $y_j \in C_j'$. Let $e \in C,f \in C'$ be the elements which give the longest cycles in $R_{x_i}, R_{y_j}$ respectively. For conjugacy class $C_i \times C_j$ of $G\times H$, and for $R_{(x_i,y_j)}$, the length of cycle of $(a,b)$ is the least common multiple of lengths of cycles of $a$ and $b$ in $R_{x_i}$ and $R_{y_j}$ respectively and hence divides the length of longest cycle in $R_{(x_i,y_j)}$ corresponding to element $(e,f)$.
\end{proof}

\begin{remark}\label{rm}
For a finite group $G$, if ${\rm Conj}(G, C)$ satisfies the Hayashi property for all conjugacy classes $C$, then $G$ is called a \textit{good} group \cite{Filip}. Moreover, a finite direct product of \textit{good} groups is also a \textit{good} group\cite[Corollary 5.1]{Filip}.
\end{remark}

In the next sections, we prove that for all finite irreducible Coxeter groups $G$, with conjugation closed subset $C\subseteq G$, ${\rm Conj}(G,C)$ satisfies the Hayashi property. In \cite{Cox1934}, Coxeter determined all finite irreducible groups of Euclidean displacements, which are generated by reflections. These groups are called Coxeter groups, which are higher dimensional analogues of dihedral groups.

The irreducible finite Coxeter groups are classified in families. The finite Coxeter groups $A_n$ are isomorphic to symmetric groups $S_{n+1},$ which are the groups of symmetries of a regular $n$-simplex in $\mathbb R^{n+1}.$ The finite Coxeter groups $B_n$ and $C_n$ are the groups of symmetries of $n$-dimensional hypercube in $\mathbb R^n,$ and its dual. The groups $D_n$ are the orientation preserving subgroups of $B_n$ and $C_n$. The finite Coxeter groups $F_4, H_3,$ and $H_4$ are the groups of symmetries of a regular $(3,4,3)$ polytope in $\mathbb R^4,$ the group of symmetries of an icosahedron and its dual dodecahedron, and the groups of symmetries of the regular $(5,3,3)$ polytope and its dual, respectively. However, the other exceptional Coxeter groups $E_6, E_7,$ and $E_8$ are not symmetry groups of regular polytopes. They are reflection groups generated by reflections in the roots of the corresponding exceptional root systems.

Apart from their geometric significance, Coxeter groups arise naturally in the theory of root systems and as Weyl groups of semisimple Lie algebras and algebraic groups. They also play an important role in the classification of semisimple Lie algebras through Dynkin diagrams. Coxeter groups have applications in algebraic and geometric combinatorics, geometry, representation theory, and mathematical physics. For further background and more details, we refer to \cite{Hum90, Hum1972, BB2005}. 

A finite Coxeter group is the direct product of finitely many irreducible Coxeter groups \cite[Section 2.2]{Hum90}. Therefore, as a consequence of Remark \ref{rm}, we will get that all finite Coxeter groups are \textit{good} groups.

%% file: Sn.tex
\section{Finite Coxeter Groups of Classical Type}\label{section.B_n}
The finite Coxeter groups of type $A_n$ are isomorphic to symmetric groups. The enumeration of conjugacy classes of $S_n$ with partitions of $n$ is well known. It is known that ${\rm Conj}(S_n, C)$ for conjugacy class $C$ of $S_n$ satisfies the Hayashi property \cite[Prop. 5.3]{Filip}. We also prove the same fact in the following theorem for the sake of completeness. The proof given here is different from the proof of \cite[Prop. 5.3]{Filip} and will be useful for the cases of $B_n$ and $D_n$.

\begin{theorem}\label{tsn}
Let $S_n$ be the symmetric group and $C$ be a conjugacy class of $S_n$. Then ${\rm Conj}(S_n,C)$ satisfies the Hayashi property.
\end{theorem}
\begin{proof}
We note from Lemma \ref{filip} that it is enough to show the existence of $e,z \in C$ such that $\langle e \rangle \cap C_{S_n}(z) \leq Z(\langle C \rangle)$. We can assume $|C|>1$. We observe that for a finite group $G$ and a conjugacy class $C$ of $G$, if $\langle C \rangle$ is an abelian group, then any two elements $e,z$ of $C$ satisfy the property $\langle e \rangle \cap C_{G}(z) \leq Z(\langle C \rangle)$. For $(1,2)^{S_3}$, we can choose $e=(1,2), z=(1,3)$. For $(1,2)^{S_4}$, $e=(1,2),z=(1,3)$, for $(1,2,3)^{S_4}$, $e=(1,2,3),z=(1,2,4)$ and for $(1,2,3,4)^{S_4}$, we can choose $e=(1,2,3,4),z=(1,2,4,3)$. Now let $n\geq 5$ and so have $Z(\langle C \rangle)=\{1\}$.

Let $\lambda_1^{\alpha_1}\lambda_2^{\alpha_2}\hdots \lambda_t^{\alpha_t}$ be the cycle structure of $e \in C$ such that $a_i$ is first element of one of $\lambda_i$-cycle.\\
\textbf{Case i}. $t\geq 3$. Let $\sigma=(a_1,a_2,a_3,\hdots,a_t)$ and $z={\sigma}e{\sigma}^{-1}$ then $z \in C$ is such that $a_{i+1}$ is the first element of one of the $\lambda_i$-cycle.
Let $k \in \mathbb N$ and $k< o(e)$ then there exists $\lambda_i$ such that $\lambda_i \nmid k$. Let one $\lambda_i$-cycle of $e$ be $(a_i,u_1,u_2,\hdots,u_{\lambda_i-1})$ and one $\lambda_{i-1}$-cycle be $(a_{i-1},x_1,x_2,\hdots,x_{\lambda_{i-1}-1})$. Let $U=\{u_1,u_2,\hdots,u_{\lambda_i-1}\}$ and $X=\{x_1,x_2,\hdots,x_{\lambda_{i-1}-1}\}$. The indices are modulo the length of the corresponding cycles with $a_0=a_t,u_0=a_i$ and $x_0=a_{i-1}$.

$ze^k(a_i)=z(u_k)={\sigma}e{\sigma}^{-1}(u_k)={\sigma}e(u_k)=\sigma(u_{k+1})=\begin{cases}
a_{i+1}\quad \text{ if } u_{k+1}=a_i\\
u_{k+1} \quad \text{ if } u_{k+1} \neq a_i
\end{cases}$ $\in \{a_{i+1}\} \cup U$

If $\lambda_{i-1}\geq 2$ then,
$e^kz(a_i)=e^k{\sigma}e{\sigma}^{-1}(a_i)=e^k{\sigma}e(a_{i-1})=e^k{\sigma}(x_1)=e^k(x_1) \in \{a_{i-1}\} \cup X$. 

Clearly $U \cap X = \{\phi\}$ and $a_{i-1} \neq a_{i+1}$ for $t\geq 3$. Thus $e^kz(a_i)\neq ze^k(a_i)$. 

If $\lambda_{i-1}=1$, then $e^kz(a_i)=e^k\sigma e\sigma^{-1}(a_i)=e^k \sigma e (a_{i-1})=e^k\sigma(a_{i-1})=e^k(a_i)=u_k \neq z(u_k)$.

\textbf{Case ii}. $t=1$. Since $n>4$, so for $e=2^{\alpha}$, we assume first three $2$-cycles of $e$ to be $(a_1,x_1)(a_2,x_2)(a_3,x_3)$ and $\sigma=(a_1,a_2,a_3)$ then for $z={\sigma}e{\sigma}^{-1}$, we see that $ez\neq ze$. 

Similarly for $e=3^{\alpha}$, we assume first two $3$-cycles of $e$ to be $(a_1,u_1,u_2)(a_2,x_1,x_2)$ and $\sigma=(a_1,a_2)$ then for $z={\sigma}e{\sigma}^{-1}$, we see that $e^kz\neq ze^k$ for $k=1,2$.

Now let $e=\lambda_1^{\alpha_1}$ and $\lambda_1 \geq 4$. If $\alpha_1=1$, then $\lambda_1=n$. We can assume the $n$-cycle is $(a_1,a_2,a_3,\hdots,a_{n})$. Let $\sigma=(a_1,a_2)$, then $z={\sigma}e{\sigma}^{-1}$ is $(a_1,a_3,a_4,\hdots,a_{n},a_2)$. Let $k=1,2,\hdots,n-1$, then\\
$e^kz(a_2)=e^k(a_1)=a_{k+1}$ and  $ze^k(a_2)=z(a_{k+2})= \begin{cases}
          a_{k+3} \quad \text{ if } 1 \leq k \leq n-3\\
          a_2 \quad \text{ if } k=n-2\\
          a_3 \quad \text{ if } k =n-1\\
\end{cases} $\\
If $\alpha_1\geq2$, then let first two $\lambda_1$-cycles be $(a_1,u_1,u_2,\hdots,u_{\lambda_1-1})$ and $(a_2,x_1,x_2,\hdots,x_{\lambda_1-1})$. Let $\sigma=(a_1,a_2)$, then first two $\lambda_1$-cycles of $z=\sigma e\sigma^{-1}$ are $(a_2,u_1,u_2,\hdots,u_{\lambda_1-1})$ and $(a_1,x_1,x_2,\hdots,x_{\lambda_1-1})$. Let $k=1,2,\hdots,\lambda_1-2$, then $e^kz(a_1)=e^k(x_1)=x_{k+1}$ and $ze^k(a_1)=z(u_k)=u_{k+1}$.
For $k=\lambda_1-1$, we have that $e^{-1}z(u_1)=e^{-1}(u_2)=u_1$ and $ze^{-1}(u_1)=z(a_1)=x_1$.

\textbf{Case iii}. $t=2$. Let $e=\lambda_1^{\alpha_1}\lambda_2^{\alpha_2}$ with $\lambda_1=1$ and let one of the $\lambda_2$-cycle be $(a_2,x_1,x_2,\hdots,x_{\lambda_2-1})$. Then for $\sigma=(a_1,a_2)$, one of the $\lambda_2$-cycle of $z={\sigma}e{\sigma}^{-1}$ is $(a_1,x_1,x_2,\hdots,x_{\lambda_2-1})$. Then, for $k=1,2,\hdots,\lambda_2-1$, we have $ze^k(a_1)=z(a_1)=x_1$ and $e^kz(a_1)=e^k(x_1)\neq x_{1}$.

Now let $\lambda_1 \geq 2$, $\lambda_2 >\lambda_1$, and one of the $\lambda_1$-cycle be $(a_1,u_1,u_2,\hdots,u_{\lambda_1-1})$ and one of the $\lambda_2$-cycle be $(a_2,x_1,x_2,\hdots,x_{\lambda_2-1})$. Then for $\sigma=(a_1,a_2)$, $z={\sigma}e{\sigma}^{-1}$ contains $(a_2,u_1,u_2,\hdots,u_{\lambda_1-1})$ and $(a_1,x_1,x_2\allowbreak,\hdots,x_{\lambda_2-2})$. Let $U=\{u_1,u_2,\hdots,u_{\lambda_i-1}\}$ and $X=\{x_1,x_2,\hdots,x_{\lambda_{i-1}-1}\}$. Then for $k=m\lambda_1 +r$ where $0\leq m < \frac{o(e)}{\lambda_1}$ and $1\leq r\leq \lambda_1-2$, we have
$ze^k(a_1)=z(u_r)=u_{r+1}\in U$ and $e^kz(a_1)=e^k(x_1)=x_{k+1} \in X$.
For $k=m\lambda_1+\lambda_1-1$, $0\leq m < \frac{o(e)}{\lambda_1}-1$, we have $ze^k(a_1)=z(u_{\lambda_1-1})=a_2$ and $e^kz(a_1)=e^k(x_1) \neq a_2$. For $k=m\lambda_1+\lambda_1$, $0\leq m< \frac{o(e)}{\lambda_1}-1$, we have $ze^k(a_1)=z(a_1)=x_1$ and $e^kz(a_1)=e^k(x_1) \neq x_1$.  We note that for $\lambda_1=2$, we only choose the last two types of $k$ values.

For $k=o(e)-1$, if $\lambda_1 \neq 2$, we have $ze^{-1}(u_1)=z(a_1)=x_1$ and $e^{-1}z(u_1)=e^{-1}(u_2)=u_1$. If $\lambda_1 = 2$, we have $ze^{-1}(u_1)=z(a_1)=x_1$ and $e^{-1}z(u_1)=e^{-1}(a_2)=x_{\lambda_2-1}$.

\end{proof}

\begin{theorem}
Let $C=\bigcup\limits_{i}C_i$ be the union of conjugacy classes of $S_n$. Then ${\rm Conj}(S_n,C)$ satisfies the Hayashi property. 
\end{theorem}
\begin{proof}
The result is checked using computations for $n\leq 4$. Let $x_i \in C_i$. We observe that $R_g \in {\rm Sym}(C)$ has the same cycle structures for all $g \in C_i$. Therefore, it is enough to show that $R_{x_i} \in {\rm Sym}(C)$ is a regular permutation. If $C_i$ is the identity class, then $R_{x_i}$ is the identity permutation. So we may assume $C_i$ is nontrivial. Let $n\geq 5$, then by using Theorem \ref{tsn}, we get that there exists $y_i \in C_i$ such that $\langle x_i \rangle \cap C_{S_n}(y_i)=\{1\}$. This implies the length of the longest cycle of $R_{x_i}\in {\rm Sym}(C_i)$, which is $o(x)$ is the same in $R_{x_i} \in {\rm Sym}(C)$. The length of other cycles corresponds to the sizes of the subgroups of $\langle x \rangle$. Hence, the result follows.
\end{proof}

Let $G=I_2(n)$, and $C$ be a conjugacy class of $G$. We recall that $I_2(n)$ is isomorphic to the dihedral group of order $2n.$ The conjugation quandle ${\rm Conj}(G,C)$ satisfies the Hayashi property \cite[Theorem 5.2]{Filip}. In the following, we prove for any conjugation closed subset of $G$. 

\begin{theorem}
Let $G=I_2(n)$ and $C=\bigcup\limits_{i} C_i$ be the union of conjugacy classes of $G$. Then ${\rm Conj}(G, C)$ satisfies the Hayashi property.
\end{theorem}

\begin{proof}
Let $G$ have the presentation $\left\{ r,s|\;r^n,s^2,srs^{-1}=r^{-1}\right\}$. Let $C=\bigcup\limits C_i$ be the union of conjugacy classes of $G$. If $C$ only contains rotations, then ${\rm Conj}(G, C)$ is trivial. If $C$ only contains reflections, then all right translations have order $2$. Now suppose $C$ contains both rotations and reflections. Again, $R_g$ for all reflections is an order $2$ permutation. If $n$ is odd, then we have $\langle r^i \rangle \cap C_{G}(s)=\{1\}$. For $n$ even, if $\{r^{n/2}\} \subset \langle r^i \rangle$, then $\langle r^i \rangle \cap C_{G}(s)=\langle r^i \rangle \cap C_{G}(sr) = \{1,r^{n/2}\}$ otherwise the intersection is $\{1\}$. Hence, we get the result.
\end{proof}

    The Coxeter groups of type $B_n$ and $C_n$ are isomorphic, and they are isomorphic to $C_2 \wr S_n,$ while the Coxeter groups of type $D_n$ is an index two subgroup of $C_2 \wr S_n.$
    
    Now we define the wreath product of the cyclic group $C_2 =\{\pm 1\} $ of order $2$ and the symmetric group $S_n$ of degree $n.$ Let $(C_2)^n$ denote the direct product of $n$-copies of $C_2, $ the group $S_n$ acts on $(C_2)^n$ by permuting its elements as follows:
     $$\pi. (a_1, a_2, \dots, a_n) = (a_{\pi^{-1}(1)}, a_{\pi^{-1}(2)}, \dots a_{\pi^{-1}(n)})$$
     for all $\pi\in S_n,$ and $(a_1, a_2, \dots, a_n)\in (C_2)^n.$ The semidirect product of $(C_2)^n$ and $S_n$ with respect to this group action is known as the wreath product of $C_2$ and $S_n.$
    This implies, for any two elements $(f, \pi) = ((a_1, a_2, \dots, a_n) ; \pi)$ and $(h, \sigma)= ((b_1, b_2, \dots, b_n) ; \sigma)$ in $C_2\wr S_n,$ their product is given by
    $$(f, \pi) (h, \sigma) = ( (a_1b_{\pi^{-1}(1)}, a_2b_{\pi^{-1}(2)}, \dots a_nb_{\pi^{-1}(n)}) ; \pi\sigma).$$

    \begin{lemma}\label{lemma_order}
    Let $(f,\pi)\in C_2 \wr S_n,$ then either the order of $(f, \pi) $ is the same as the order of $\pi,$ or twice the order of $\pi.$
    \end{lemma}
    \begin{proof}
    Clearly $(f,\pi)^{o(\pi)}\in (C_2)^n,$ where $o(\pi)$ denotes the order of $\pi$ and $(C_2)^n:= \{ (h, 1)\;|\; h \in (C_2)^n\}.$ This implies that the order of  $(f,\pi)^{o(\pi)}$ is either $1$ or $2,$ and hence the result. 
    \end{proof}

    Now, we discuss the enumeration of conjugacy classes of $C_2 \wr S_n$ with signed partitions of $n$. We begin by recalling the definition of a signed partition of $n.$
    
    \begin{definition} 
    Let $\lambda =\lambda_1^{\alpha_1} \lambda_2^{\alpha_1} \dots \lambda_t^{\alpha_t}$ be a partition of $n.$ A signed partition of $n$ of type $\lambda$  is a partition in which all parts $\lambda_i$ appear with two signs,
    denoted by $\lambda_i$ and $\overline{\lambda}_i.$ We call the parts $\lambda_i$ positive parts and $\overline{\lambda}_i$ negative parts.
    \end{definition}
    
    For example, $11, 1\overline{1},$ and $ \overline{1}\overline{1}$ are all signed partitions of $2$
     of type $11.$
     For an element $(f,\pi)= ((a_1, a_2, \dots, a_n) ; \pi)\in C_2 \wr S_n,$ the cycle product of $(f, \pi)$ corresponding to a cycle $(j, \pi(j), \dots, \pi^{t-1}(j))$ in $\pi$ is given by $a_ja_{\pi(j)} \dots a_{\pi^{t-1}(j)}\in C_2.$ Let $\overline{\lambda}$ be the signed partition of $n$ of type $\lambda$ associated with conjugacy class of $(f,\pi).$ Clearly, as the cycle  $(j, \pi(j), \dots, \pi^{t-1}(j))$ lie in $\pi,$ $t$ is a part in partition $\lambda$ of $n.$ If the cycle product  $a_ja_{\pi(j)} \dots a_{\pi^{t-1}(j)}$ is equal to $-1$, then the part of $\overline{\lambda}$ corresponding to this cycle is $\overline{t}, $ otherwise it is $t.$ For more details, we refer the reader to \cite{coxeter_KS}. 
     
     We fix some notation now. We denote the $m$-tuple with all entries equal to $1$ by ${\bf 1}_m$, and $\bar{\bf 1}_m$ denotes the $m$-tuple with the first entry equal to $-1$ and all other entries equal to $1$. We use $-{\bf1}_m$ to denote the $m$-tuple $(-1,-1,-1,\hdots,-1)$.
     
    \begin{lemma}\label{Lemma_order}
    Let $(f,\pi)\in C_2 \wr S_n,$ and let $\overline{\lambda} := \lambda_1^{\alpha_1}\lambda_2^{\alpha_2}\dots \lambda_k^{\alpha_k} \overline{\mu}_1^{\beta_1}\overline{\mu}_2^{\beta_2}\dots \overline{\mu}_l^{\beta_l}$ be the signed partition of $n$ associated with the conjugacy class of $(f,\pi).$ The order of $(f, \pi)$ is twice the order of $\pi$ if and only if there is at least one negative part $\overline{\mu}_i$ of $\overline{\lambda}$ such that $2^a| \mu_i$, where $a\geq0$ is the highest power of $2$ such that $2^a$ divides some part of partition $\overline{\lambda}$. 
    \end{lemma}
    
    \begin{proof}
    Let $o(x)$ denote the order of the element $x.$ If $l=0,$ then it is easy to see that $o((f, \pi))=o(\pi).$ We suppose $l>0.$ As the signed partition associated with $(f, \pi)$ is $\overline{\lambda},$ let 
    \[ (f, \pi) = ( {\bf 1}_{\lambda_1 \alpha_1}, \dots , {\bf 1}_{\lambda_k \alpha_k}, \bar{\bf 1}_{\mu_1 \beta_1}, \dots \bar{\bf 1}_{\mu_l \beta_l} ; \pi)\]

    Let $a (\geq 0)$ be the highest power of $2$ such that $2^a$ divides some part of the partition $\overline{\lambda}$. Suppose there are negative parts in $\overline{\lambda}$ that are divisible by $2^a$, then we rearrange the parts of the signed partition in such a way that $2^a |\mu_j$ for all $i\leq j\leq l.$ Now we compute
    \[ (f, \pi)^{o(\pi)} = ( {\bf 1}_{\lambda_1 \alpha_1}, \dots , {\bf 1}_{\lambda_k \alpha_k}, {\bf 1}_{\mu_1 \beta_1}, \dots {\bf 1}_{\mu_{i-1} \beta_{i-1}},-{\bf 1}_{\mu_i \beta_i}, \dots, -{\bf 1}_{\mu_l \beta_l} ;  (1) )\]
    
    So in this case, we have $o((f, \pi))=2o(\pi).$ If there is no negative part in $\overline{\lambda}$, that is divisible by $2^a$, then $o(f,\pi)=o(\pi)$.
    \end{proof}

Now, we show that conjugation quandles over all finite Coxeter groups of type $B_n$ and $D_n$ satisfy the Hayashi property. We begin this section by describing the normal subgroups of $C_2 \wr S_n$ generated by its conjugacy classes. 

Recall that we denote the cyclic group of order $2$ with $C_2,$ and $(C_2)^n$ denotes the direct product of $n$ copies of $C_2.$ The group $(C_2)^n_0=\{ (a_1, a_2, \dots, a_n) \;|\; \Pi_{i=1}^n a_i =1\}$ is an index $2$ subgroup of $(C_2)^n.$ The finite Coxeter group of type $D_n = \{ ((a_1, a_2, \dots, a_n); \pi) \;|\; \Pi_{i=1}^n a_i =1, \pi \in S_n\}$ is an index $2$ subgroup of $C_2 \wr S_n.$ We also denote this by $(C_2)_0^n \rtimes S_n.$ The another two index $2$ subgroups of $C_2 \wr S_n$ generated by its conjugacy classes are $C_2 \wr A_n$ and $(C_2)_1^n\rtimes S_n.$ The subgroup $C_2 \wr A_n = \{ ((a_1, a_2, \dots, a_n); \pi) \;|\; \pi \in A_n\}$ and the finite presentation of subgroup $(C_2)_1^n\rtimes S_n$ is given below:
\[ (C_2)_1^n\rtimes S_n = \{ ((a_1, a_2, \dots, a_n); \pi) \;|\; \Pi_{i=1}^n a_i =-1, \pi \in S_n \setminus A_n\} \cup\{ ((a_1, a_2, \dots, a_n); \pi) \;|\; \Pi_{i=1}^n a_i =1, \pi \in A_n\} \]
\begin{table}[H]
\caption{A list of normal subgroups generated by the conjugacy classes of $C_2\wr S_n.$ }
\label{table}
\begin{center}
\begin{tabular}{|c|c|c|}
\hline
   Representative of $C$ & $\langle C \rangle$ & $Z(\langle C \rangle)$ \\
   \hline
  $ ({\bf 1}_n; (1)) $ & $\{ ({\bf 1}_n; (1))\}$  & $\{ ({\bf 1}_n; (1))\}$\\
  \hline
  $ ({-\bf 1}_n; (1)) $ & $Z(C_2 \wr S_n) \cong C_2$ & $Z(C_2 \wr S_n) \cong C_2$\\
  \hline
   $ ({\bf 1}_{n-m},{-\bf  1}_{m}; (1)), ~~ m$ even &  $(C_2)^n_0$ & $(C_2)^n_0$\\
   \hline
   $ ({\bf 1}_{n-m},{-\bf  1}_{m}; (1)),~~ m$ odd &  $(C_2)^n$ & $(C_2)^n$\\
   \hline
    $ ({\bf 1}_{n-m},-{\bf  1}_{m}; \pi), ~~ m$ even, $\pi$ odd &  $(C_2)^n_0\rtimes S_n $ & $Z(C_2 \wr S_n) \cong C_2$ if $n$ even, trivial otherwise\\
   \hline
    $ ({\bf 1}_{n-m},-{\bf  1}_{m}; \pi), ~~ m$ odd, $\pi$ odd  &  $(C_2)^n_1 \rtimes S_n$ & $Z(C_2 \wr S_n) \cong C_2$ if $n$ even, trivial otherwise\\
   \hline
    $ ({\bf 1}_{n-m},-{\bf  1}_{m}; \pi), ~~ m$ even, $\pi$ even &  $(C_2)^n_0\rtimes A_n $ & $Z(C_2 \wr S_n) \cong C_2$ if $n$ even, trivial otherwise\\
   \hline
    $ ({\bf 1}_{n-m},-{\bf  1}_{m}; \pi), ~~ m$ odd, $\pi$ even &  $(C_2)^n\rtimes A_n $ &  $Z(C_2 \wr S_n) \cong C_2$\\
   \hline
\end{tabular}
\end{center}
\end{table}

We begin the proof for $B_n$ by dividing into the cases of $o(f,\pi)=o(\pi)$ or $o(f,\pi)=2o(\pi)$. We only need to consider conjugacy classes $C$ with nonabelian $\langle C \rangle$. 

\begin{lemma}\label{lc1}
Let $C$ be a conjugacy class of $(f,\pi) \in C_2 \wr S_n,n\geq5$ such that $o((f,\pi))=o(\pi)$. Then there exists $x,y \in C$ such that $\langle x \rangle \cap C_{C_2 \wr S_n}(y)=\{1\}$.
\end{lemma}

\begin{proof}
Let $x=(f,\pi) \in C$. For $\pi \in S_n$, using Theorem \ref{tsn} we can construct $z=\sigma{\pi}\sigma^{-1} \in S_n$ such that $\langle \pi \rangle \cap C_{S_n}(z) = \{1\}$. Now we construct $y=(h,z)=(h',\sigma)(f,\pi)(h',\sigma)^{-1}$ where $h' \in C_2^{n}$. Then for $k=1,2,\hdots,o(\pi)-1$, we have $\pi^kz\neq z\pi^k$ and hence $x^ky\neq yx^k$.
\end{proof}

\begin{lemma}\label{lc2}
Let $\bar{\lambda}=\overline{\mu}_1^{\beta_1}\overline{\mu}_2^{\beta_2}\hdots \overline{\mu}_l^{\beta_l}$ be the signed partition of conjugacy class $C$ of $C_2 \wr S_n,n\geq5$ satisfying for some $a\in \mathbb{Z}_{\geq 0}$, $\gcd(\frac{\mu_i}{2^a},2)=1$ for all $1\leq i\leq l$. Then there exists $x,y \in C$ such that $\langle x \rangle \cap C_{C_2\wr S_n}(y)=Z(\langle C \rangle)$.
\end{lemma}

\begin{proof} 
Let $x=(f,\pi) \in C$. We observe that $o(x)=2o(\pi)$, $(f,\pi)^{o(\pi)}=({-\bf 1}_n;(1) )$ and $Z(\langle C \rangle)=\{({\bf 1}_n;(1)),(-{\bf1}_n;(1))\}$. Using Theorem \ref{tsn}, we can construct $z=\sigma\pi\sigma^{-1}$ such that $\langle \pi \rangle \cap C_{S_n}(z)=\{1\}$. Now we construct $y=(h,z)=(h',\sigma)(f,\pi)(h',\sigma)^{-1}$ with $h'\in C_2^n$. Then for $k=1,2,\hdots,2o(\pi)-1$, $k\neq o(\pi)$, we have $\pi^kz\neq z\pi^k$ and hence $x^ky\neq yx^k$. For $k=o(\pi)$, we have $x^k =({-\bf 1}_n;(1) )$ and so $x^ky=yx^k$.
\end{proof}

\begin{lemma}\label{lc3}
Let $C$ be a nontrivial conjugacy class of $ C_2 \wr S_n$, $n\geq 5$, except for the above two cases. Then there exist $x,y \in C$ such that $\langle x \rangle \cap C_{C_2 \wr S_n}(y)=\{1\}$.
\end{lemma}

\begin{proof}
Let $\bar{\lambda}=\lambda_1^{\alpha_1} \overline{\mu}_1^{\beta_1}\lambda_2^{\alpha_2} \overline{\mu}_2^{\beta_2}\hdots \lambda_k^{\alpha_k} \overline{\mu}_l^{\beta_l}$ be the signed partition of conjugacy class of element $x=(f,\pi)\in C_2 \wr S_n$, such that 
\[(f,\pi)= ({ {\bf 1}_{\lambda_1},\cdots,{\bf 1}_{\lambda_1}, {\bf \bar 1}_{\mu_1},\cdots,{\bf \bar 1}_{\mu_1},\cdots,{\bf 1}_{\lambda_k},\cdots,{\bf 1}_{\lambda_k},{\bf \bar 1}_{\mu_l},\cdots,{\bf \bar 1}_{\mu_l}}; \pi)\]

where ${\bf 1}_{\lambda_i}=(1,1,\hdots,1)$, is $\lambda_i$-tuple corresponding to $\lambda_i$-cycle and ${\bf \bar 1}_{\mu_i}=(-1,1,1,\hdots,1)$ is $\mu_i$-tuple corresponding to $\mu_i$-cycle of $\pi$. We note to exclude cases of Lemma \ref{lc1} and Lemma \ref{lc2} for $\bar{\lambda}$. \\
\textbf{Case i}. Let $k\geq 1$ and $\bar{\lambda}$ is arranged such that $ \lambda_1 \geq 2$ and ${\bf -1}_{\mu_1}$ of $(f,\pi)$ in $(f,\pi)^{o(\pi)}$ becomes $(-1,-1,\hdots,-1)$ and $\lambda_1 \neq \mu_1$. Let $(a_1,u_1,u_2,\hdots,u_{\lambda_1-1})$ be one of the $\lambda_1$- cycle, and $(a_2,x_1,x_2,\hdots,x_{\mu_1-1})$ be one of the $\mu_1$-cycle of $\pi$. Let $U=\{u_1,u_2,\hdots,u_{\lambda_1-1}\}$ and $X=\{x_1,x_2,\hdots,x_{\mu_1-1}\}$. Now by using Theorem \ref{tsn}, we can construct an appropriate $\sigma=(a_1,a_2,\hdots,a_m)$ such that for $z=\sigma\pi\sigma^{-1}$ we have $\langle \pi \rangle \cap C_{S_n}(z)=\{1\}$. Let $(h,z)=(h',\sigma)(f,\pi)(h',\sigma)^{-1}$ where $h' \in C_2^n$. It is enough to show $(f,\pi)^{o(\pi)} \notin C_{C_2\wr S_n}((h,z))$.
Let $(f,\pi)^{o(\pi)}=(g,1)$. Then we have $g_{a_2}=-1$. If $h_{a_2}=-1$, then in $(g,1)(h,z)$ we have $g_{a_2}h_{a_2}=1$, and in $(h,z)(g,1)$ we have $h_{a_2}g_{\sigma\pi^{-1}\sigma^{-1}(a_2)}=-1. g_{\sigma\pi^{-1}(a_1)}=-1.g_{\sigma(u_{\lambda_1-1})}=-1.g_{u_{\lambda_1-1}}=-1.1 =-1$. If $h_{a_2}=1$, then in $(g,1)(h,z)$ we have $g_{a_2}h_{a_2}=-1$, and in $(h,z)(g,1)$ we have $h_{a_2}g_{\sigma\pi^{-1}\sigma^{-1}(a_2)}=1. g_{\sigma\pi^{-1}(a_1)}=1.g_{\sigma(u_{\lambda_1-1})}=1.g_{u_{\lambda_1-1}}=1.1 =1$\\
If $k=1$ and $\lambda_1=1$ then arrange $\bar{\lambda}=\bar{\mu_1}^{\beta_1}1^{\alpha_1}\bar{\mu_2}^{\beta_2}\hdots\bar{\mu_l}^{\beta_l}$  such that  ${\bf \bar 1}_{\mu_1}$ of $(f,\pi)$ in $(f,\pi)^{o(\pi)}$ becomes $(-1,-1,\hdots,-1)$. Let $(a_2,x_1,x_2,\hdots,x_{\mu_1-1})$ be one of the $\mu_1$-cycle of $\pi$ and $(a_1)$ be $1$-cycle. Then $\sigma:=(a_2,a_1,...)$ and for $z=\sigma\pi\sigma^{-1}$, we have $\langle \pi \rangle \cap C_{S_n}(z)=\{1\}$. Let $(h,z)=(h',\sigma)(f,\pi)(h',\sigma)^{-1}$ where $h' \in C_2^n$.
Let $(f,\pi)^{o(\pi)}=(g,1)$. Then we have $g_{a_1}=1$. If $h_{a_1}=-1$ then in $(g,1)(h,z)$ we have $g_{a_1}h_{a_1}=-1$ and in $(h,z)(g,1)$ we have $h_{a_1}g_{\sigma\pi^{-1}\sigma^{-1}(a_1)}=-1.g_{\sigma\pi^{-1}(a_2)}=-1.g_{\sigma(x_{\mu_1-1})}=-1.g_{x_{\mu_1-1}}=-1.-1=1$. If $h_{a_1}=1$ then in $(g,1)(h,z)$ we have $g_{a_1}h_{a_1}=1$ and in $(h,z)(g,1)$ we have $h_{a_1}g_{\sigma\pi^{-1}\sigma^{-1}(a_1)}=1.g_{\sigma\pi^{-1}(a_2)}=1.g_{\sigma(x_{\mu_1-1})}=1.g_{x_{\mu_1-1}}=1.-1=-1$.

\textbf{Case ii.} $k=1$, $l\geq 2$ and let $\bar{\lambda}=\bar{\mu_2}^{\beta_2}\bar{\mu_1}^{\beta_1}\lambda_1^{\alpha_1}\bar{\mu_3}^{\beta_3}\hdots\bar{\mu_l}^{\beta_l}$ signed partition of conjugacy class of element $(f,\pi)\in C_2 \wr S_n$, be arranged such that $\mu_2\neq 1$ and ${\bf \bar1}_{\mu_2}$ of $(f,\pi)$ in $(f,\pi)^{o(\pi)}$ becomes $\bf{1_{\mu_2}}$ $=(1,1,\hdots,1)$ and ${\bf \bar1}_{\mu_1}$ of $(f,\pi)$ in $(f,\pi)^{o(\pi)}$ becomes $(-1,-1,\hdots,-1)$ and $\lambda_1=\mu_1$. We note that this arrangement implies $\mu_1 \neq 1$ and so $\lambda_1\neq 1$. Let $(a_1,u_1,u_2,\hdots,u_{\mu_2-1})$ be one of the $\mu_2$- cycle, and $(a_2,x_1,x_2,\hdots,x_{\mu_1-1})$ be one of the $\mu_1$-cycle of $\pi$. Now the proof works similarly to the previous case.

\textbf{Case iii}. $k=0$, $l\geq 2$ and let $\bar{\lambda}=\bar{\mu_2}^{\beta_2}\bar{\mu_1}^{\beta_1}\bar{\mu_3}^{\beta_3}\hdots\bar{\mu_l}^{\beta_l}$ signed partition of conjugacy class of element $(f,\pi)\in C_2 \wr S_n$, be arranged such that ${\bf \bar 1}_{\mu_2}$ of $(f,\pi)$ in $(f,\pi)^{o(\pi)}$ becomes $\bf{1_{\mu_2}}$ $=(1,1,\hdots,1)$ and ${\bf \bar 1}_{\mu_1}$ of $(f,\pi)$ in $(f,\pi)^{o(\pi)}$ becomes $(-1,-1,\hdots,-1)$. Let $(a_1,u_1,u_2,\hdots,u_{\mu_2-1})$ be one of the $\mu_2$-cycle, and $(a_2,x_1,x_2,\hdots,x_{\mu_1-1})$ be one of the $\mu_1$-cycle of $\pi$. Now the proof works similarly to case (i).

\textbf{Case iv.} Let $\bar{\lambda}=2^{\alpha_1}\bar{2}^{\beta_1}$ be the signed partition of conjugacy class of element $e=(f,\pi)\in C_2 \wr S_n$. If $\alpha_1\geq2$, then let $(a_1,x_1)(a_2,x_2)$ be first two cycles of $2^{\alpha_1}$ and let $(a_3,x_3)$ be one of the $2$- cycle of $\bar{2}^{\beta_1}$. Then for $\sigma=(a_1,a_2,a_3)$ and $z=\sigma\pi\sigma^{-1}$, we have $\langle \pi \rangle \cap C_{S_n}(z)=\{1\}$. We have $o(e)=4$ and $o(\pi)=2$. Let $(h,z)=(h',\sigma)(f,\pi)(h',\sigma)^{-1}$ for $h' \in C_2^n$. Let $(f,\pi)^2=(g,1)$. Then we have $g_{a_3}=-1$. If $h_{a_3}=1$, then in $(g,1)(h,z)$ we have $g_{a_3}h_{a_3}=-1$, and in $(h,z)(g,1)$ we have $h_{a_3}g_{\sigma\pi^{-1}\sigma^{-1}(a_3)}=1.g_{\sigma\pi^{-1}(a_2)}=1.g_{\sigma(x_2)}=1.g_{x_2}=1.1=1$. 
If $h_{a_3}=-1$, then in $(g,1)(h,z)$ we have $g_{a_3}h_{a_3}=-1.-1=1$, and in $(h,z)(g,1)$ we have $h_{a_3}g_{\sigma\pi^{-1}\sigma^{-1}(a_3)}=-1.g_{\sigma\pi^{-1}(a_2)}=-1.g_{\sigma(x_2)}=-1.g_{x_2}=-1.1=-1$. The case of $\alpha_1=1, \beta_1\geq2$ works similarly.

Now let $\bar{\lambda}=\lambda_1^{\alpha_1}\bar{\lambda_1}^{\beta_1}$ with $\lambda_1\geq 3$ be the signed partition of conjugacy class of element $e=(f,\pi)\in C_2 \wr S_n$.
Let $(a_1,u_1,u_2,\hdots,u_{\lambda_1-1})$ be one of the $\lambda_1$-cycle corresponding to the part$\lambda_1^{\alpha_1}$ and $(a_2,x_1,x_2,\hdots,x_{\lambda_1-1})$ be one of the $\lambda_1$-cycle corresponding to $\bar{\lambda_1}^{\beta_1}$ part. Then for $\sigma=(a_1,a_2)$ and $z=\sigma\pi\sigma^{-1}$, we have $\langle \pi \rangle \cap C_{S_n}(z)=\{1\}$. We have $o(e)=2\lambda_1$ and $o(\pi)=\lambda_1$. Let $(h,z)=(h',\sigma)(f,\pi)(h',\sigma)^{-1}$ for $h' \in C_2^n$. Let $(f,\pi)^{\lambda_1}=(g,1)$. Then we have $g_{a_2}=-1$. If $h_{a_2}=1$, then in $(g,1)(h,z)$ we have $g_{a_2}h_{a_2}=-1$, and in $(h,z)(g,1)$ we have $h_{a_2}g_{\sigma\pi^{-1}\sigma^{-1}(a_2)}=1.g_{\sigma\pi^{-1}(a_1)}=1.g_{\sigma(u_{\lambda_1-1})}=1.g_{u_{\lambda_1-1}}=1.1=1$. 
If $h_{a_2}=-1$, then in $(g,1)(h,z)$ we have $g_{a_2}h_{a_2}=-1.-1=1$, and in $(h,z)(g,1)$ we have $h_{a_2}g_{\sigma\pi^{-1}\sigma^{-1}(a_2)}=-1.g_{\sigma\pi^{-1}(a_1)}=-1.g_{\sigma(u_{\lambda_1-1})}=-1.g_{u_{\lambda_1-1}}=-1.1=-1$.
\end{proof}

\begin{theorem}
Let $C$ be a conjugacy class of the finite Coxeter group $B_n$. Then there exists $x,y \in C$ such that $\langle x \rangle \cap C_{B_n}(y) \subseteq Z(\langle C \rangle)$.
\end{theorem}

\begin{proof}
The result is checked by computations for $B_n$ with $n\leq 4$. For $B_n$ with $n\geq 5$, the result follows using Lemma \ref{lc1}, Lemma \ref{lc2} and Lemma \ref{lc3}. 
\end{proof}

\begin{corollary}
 Let $C$ be a conjugacy class of the finite Coxeter group $D_n$. Then there exists $x,y \in C$ such that $\langle x \rangle \cap C_{D_n}(y) \subseteq Z(\langle C \rangle)$.
\end{corollary}

\begin{proof}
The result is checked by computations for $D_n$ with $n\leq 4$. For $n\geq 5$, let $x=(f,\pi) \in D_n$. In the proof of Lemma \ref{lc1}, Lemma \ref{lc2} and Lemma \ref{lc3} we can take $h'={\bf1}_n \in C_2^n$. Then $(h',\sigma) \in D_n$ and so $y=(h',\sigma)(f,\pi)(h',\sigma)^{-1}$. We note that for $n$ odd, $D_n$ does not have conjugacy classes $C$ satisfying Lemma $\ref{lc2}$. Hence, the result follows. 
\end{proof}

Now we will give the proof for $C=\bigcup\limits_{i}C_i$. The following general lemma will be useful for us for the cases of $B_n$ and $D_n$.

\begin{lemma}\label{union}
Let $C=\bigcup\limits_{i}C_i$ with $x_i \in C_i$ be the union of conjugacy classes of group $G$. If ${\rm Conj}(G,C_i)$ satisfies the Hayashi property and $Z(\langle  C_i\rangle) \leq Z(\langle C \rangle)$ then $R_{x_i} \in {\rm Sym}(C)$ is a regular permutation. 
\end{lemma}

\begin{proof}
Given the assumption, there are $e,z \in C_i$ such that $\langle e \rangle \cap C_{G}(z) \leq Z(\langle C_i\rangle) \leq Z(\langle C \rangle)$. Therefore $R_e \in {\rm Conj}(G,C)$ is a regular permutation using Lemma \ref{ue}
\end{proof}

\begin{theorem}\label{unionbn}
Let $C=\bigcup\limits_{i}C_i$ be the union of conjugacy classes of $G=B_n$(or $D_n$). Then ${\rm Conj}(G, C)$ satisfies the Hayashi property. 
\end{theorem}

\begin{proof}
The result is checked using computations for $B_n, D_n$ with $n\leq 4$. Let $n\geq 5$ and $x_i \in C_i$. If $C_i=\{{(\bf 1}_n; (1))\}$ or $\{{(-\bf 1}_n; (1))\}$ then $R_{x_i} \in {\rm Sym}(C)$ is the identity permutation. So for the nontrivial case, we can assume that for some $j$, the subgroup $\langle C_j \rangle$ is nonabelian. It is enough to show that $R_{x_i} \in {\rm Sym}(C)$ is a regular permutation. Using Table \ref{table}, we get if $\langle C_i \rangle$ is nonabelian then $ Z(\langle C_i \rangle) \leq Z(\langle C \rangle)$ so $R_{x_i} \in {\rm Sym}(C)$ is a regular permutation using Lemma \ref{union}. If $\langle C_i \rangle$ is abelian and $C_i$ is noncentral, then $o(x_i)=2$ and $R_{x_i} \in {\rm Sym}(C)$ has order $2$ due to existence of $C_j$ with $\langle C_j \rangle$ nonabelian. Hence, the result follows. 
\end{proof}

In the context of Lemma \ref{union}, we note that the conjugacy classes of a finite group do not always satisfy $Z(\langle C_i \rangle) \leq Z(\langle C\rangle )$. However, that does not mean $R_{x_i} \in {\rm Conj}(G, C)$ is not regular. As an example, we have the Theorem \ref{unionbn} with a non central $C_i$ such that $\langle C_i \rangle$ is abelian.

Two conjugacy classes of a group $G$ are called $z$-conjugate if the centralizers of their representatives are conjugate subgroups. The $z$-conjugate conjugacy classes of symmetric groups and finite Coxeter groups are determined in \cite{BKS} and \cite{coxeter_KS}, respectively. In the following, we show that the union of $z$-conjugate conjugacy classes satisfies the property of Lemma \ref{union}.

\begin{lemma}\label{zconjugate}
Let $C=\bigcup\limits_{i=1}^{k}C_i$ be the union of distinct pairwise  $z$-conjugate (or $z$-equivalent) conjugacy classes. Then $Z(\langle C_i \rangle) \subseteq Z(\langle C\rangle)$ for all $1 \leq i \leq k$.
\end{lemma}
\begin{proof}
Let $g_i \in C_i$. For each $j \neq i$, as $C_i$ and $C_j$ are $z$-conjugate, we can choose a representative $g_j \in C_j$ such that $C_G(g_i) = C_G(g_j)$. 

Let $x \in Z(\langle C_i \rangle)$, then $xg = gx$ for all $g \in C_i$, i.e., $x zg_iz^{-1} = zg_iz^{-1}x$ for all $z \in G$. Thus $x \in \bigcap_{z \in G} C_G(zg_iz^{-1})$. But we know that $C_G(zg_iz^{-1}) = zC_G(g_i)z^{-1}$ for all $\in G$. Hence, \[\bigcap_{z \in G} C_G(zg_iz^{-1}) = \bigcap_{z \in G} zC_G(g_i)z^{-1} = \bigcap_{z \in G} zC_G(g_j)z^{-1} = \bigcap_{z \in G}C_G(zg_jz^{-1})\] 

Thus, we have $x \in \bigcap_{z \in G}C_G(zg_jz^{-1})$ and we know that $\langle C_i\rangle \subseteq \langle C\rangle$, so, we get that $x \in Z(\langle C\rangle)$. 
\end{proof}

\begin{theorem}
Let $\{C_i\; ; i\in I\}$ be the set of pairwise $z$-conjugate conjugacy classes such that ${\rm Conj}(G,C_i)$ satisfies the Hayashi property for all $i \in I$. Then for $C=\bigcup\limits_{i\in I} C_i$, ${\rm Conj}(G,C)$ satisfies the Hayashi property.
\end{theorem}

\begin{proof}
The proof follows from Lemma \ref{union} and Lemma \ref{zconjugate}.
\end{proof}

%% file: Exceptional.tex
\section{Exceptional Coxeter Groups}\label{section.exceptional}
In this section, we show that for the exceptional Coxeter groups $G$, namely $E_6, E_7, E_8, F_4, H_3$, and $H_4$, with conjugation closed subset $C$, ${\rm Conj}(G,C)$ satisfies the Hayashi property. We note that for an exceptional group $G$ with $C$ as its conjugacy class, the subgroup $Z(\langle C \rangle)$ is trivial or equal to $Z(G)$. Moreover, we have $Z(N)=\{1\}$ or $\{\pm 1\}$ for all normal subgroups $N$ of $G$, where $-1$ represents the central element of order $2$. Therefore, we can implement the following \texttt{GAP} \cite{GAP4} function for verification.

\begin{center}
\begin{framed}
\begin{verbatim}
# Input a list of non central conjugacy class [c1,c2,...] of a pre-defined group G 
# The function returns a list of {e,z} satisfying the required property.

SuitablePairs:=function(classes)
local u,A,c,e,ord,m,powers,z,IsSuitable,i;
u:=Elements(Union(classes));
A:=[];
for c in classes do
e:=Representative(c);
ord:=Order(e);
m := ord/2;
if IsOddInt(ord) or ord=2 then
    powers:=[1..ord-1];
elif IsEvenInt(ord) and e^m in Center(G) then
    powers:=Difference([1..ord-1],[m]);
else
    powers:=[1..ord-1];
fi;
for z in u do
IsSuitable:=true;
for i in powers do
if e^i*z=z*e^i then
IsSuitable:=false;
break;
fi; od;
if IsSuitable=true then
Add(A,[e,z]);
break;
fi; od; od;
return A;
end;
\end{verbatim}
\end{framed}
\end{center}

\begin{theorem}\label{E}
Let $G$ be the Coxeter group $E_6,E_7,E_8,F_4,H_3,H_4$ and $C$ be a conjugacy class of $G$, then ${\rm Conj}(G,C)$ satisfies the Hayashi property. 

\end{theorem}
\begin{proof}
Let $C$ be a conjugacy class of $G$. We can assume $C$ to be such that $\langle C \rangle$ is nonabelian. We can identify $e,z \in C$ satisfying $\langle e \rangle \cap C_G(z) \leq Z(\langle C \rangle)$ in a few seconds using the above code for conjugacy classes of all exceptional groups except for $E_8$. For the case of $E_8$, instead of searching for $z$ in $C$, we search in the subset $\{geg^{-1}; g \in D\} \subset C$, where D is a conjugacy class of size $4480$, and we can find the desired pair.
\end{proof}

\begin{theorem}
Let $G$ be one of the Coxeter groups $E_6,E_7,E_8,F_4,H_3,H_4$, and $C$ be the conjugation closed subset of $G$. Then ${\rm Conj}(G,C)$ satisfies the Hayashi property.
\end{theorem}
\begin{proof}
Let $C=\bigcup\limits_{i}C_i$ be the union of conjugacy classes. If $Z(G)=\{1\}$ and  $\langle C_i \rangle$ is abelian, then $C_i=\{1\}$ and if $Z(G)=\{\pm 1\}$ and $\langle C_i \rangle$ are abelian, then $C_i=\{1\}$ or $\{-1\}$. For the nontrivial case, we can assume that there is at least one $C_i$ with nonabelian $\langle C_i \rangle$.  If $\langle C_i \rangle$ is nonabelian, then the same pair $e,z \in C_i$ from Theorem \ref{E}, also works for $C$ using Theorem \ref{union}. If $\langle C_i \rangle$ is abelian, then for $x_i \in C_i$, $R_{x_i} \in {\rm Sym}(C)$ is identity permutation.
\end{proof}

We have proved that for all finite irreducible Coxeter groups $G$, with a conjugacy closed subset $C$, ${\rm Conj}(G,C)$ satisfies the Hayashi property. However all groups do not satisfy this property (see Example \ref{example}). Therefore, we state the following problem.

\begin{problem}
Classify finite groups $G$ such that ${\rm Conj}(G,C)$ satisfies the Hayashi property for all conjugation closed subsets $C$ of $G$.
\end{problem}

\medskip \noindent \textbf{Acknowledgments.}
The third named author acknowledges the support of the IISER Mohali institute postdoctoral fellowship during this work.